\documentclass[reqno]{amsart}
\usepackage{amsmath,amsthm,amscd,amssymb,amsfonts, amsbsy}
\usepackage{latexsym, color, enumerate}
\usepackage{pxfonts}

\usepackage[hypertex]{hyperref}

\theoremstyle{plain}
\newtheorem{theorem}{Theorem}[section]
\newtheorem{lemma}[theorem]{Lemma}
\newtheorem{corollary}[theorem]{Corollary}
\newtheorem{proposition}[theorem]{Proposition}

\theoremstyle{definition}
\newtheorem{definition}[theorem]{Definition}

\newtheorem{assumption}[theorem]{Assumption}

\theoremstyle{remark}
\newtheorem{remark}[theorem]{Remark}

\newcommand{\dv}{\operatorname{div}}

\newcommand{\dist}{\operatorname{dist}}

\numberwithin{equation}{section}

\newcommand{\bR}{\mathbb{R}}
\newcommand{\bS}{\mathbb{S}}

\newcommand{\vq}{\vec{q}}

\newcommand\cA{\mathcal{A}}

\providecommand{\set}[1]{\{#1\}}

\providecommand{\abs}[1]{\lvert#1\rvert}
\providecommand{\Abs}[1]{\left\lvert#1\right\rvert}

\providecommand{\norm}[1]{\lVert#1\rVert}
\providecommand{\Norm}[1]{\left\lVert#1\right\rVert}

\renewcommand{\vec}[1]{\boldsymbol{#1}}

\def\XXint#1#2#3{{\setbox0=\hbox{$#1{#2#3}{\int}$}
		\vcenter{\hbox{$#2#3$}}\kern-.5\wd0}}

\newcommand{\p}{\partial}
\newcommand{\epsi}{\varepsilon}

\begin{document}

%
%

	\title[Oblique Derivative Problem]{$C^2$ estimate for oblique derivative problem\\ with mean Dini coefficients}
	
	\author[Hongjie Dong]{Hongjie Dong}	
	
	\address{Brown University\\ Division of Applied Mathematics \\
		Providence RI 02906}
	
	\email{hongjie$\_$dong@brown.edu}
\thanks{H. Dong and Z. Li were partially supported by the NSF under agreement DMS-1600593.}
	
	\author[Zongyuan Li]{Zongyuan Li}
	
	\address{Brown University\\ Division of Applied Mathematics \\
		Providence RI 02906}
	
	\email{zongyuan$\_$li@brown.edu}
	
\begin{abstract}
We consider second-order elliptic equations in non-divergence form with oblique derivative boundary conditions. We show that any strong solutions to such problems are twice continuously differentiable up to the boundary provided that the mean oscillations of coefficients satisfy the Dini condition and the boundary is locally represented by a $C^1$ function whose first derivatives are Dini continuous. This improves a recent result in \cite{DLK}. An extension to fully nonlinear elliptic equations is also presented.
\end{abstract}
\maketitle


\section{Introduction and Main Results}
In this paper, we consider strong solutions $u\in W^2_2(\Omega)$ to the oblique derivative problem in a bounded domain $\Omega\subset \bR^d,d\ge 2$:
\begin{equation}\label{eqn-main-with-lower-order}
\begin{cases}
Lu := -a_{ij}D_{ij}u + b_i D_i u + cu = f &\text{in}\,\Omega,\\
Bu := \beta_0 u +\beta_iD_i u = g &\text{on}\,\p\Omega.
\end{cases}
\end{equation}
For simplicity, the notation
$$\vec A:=(a_{ij})_{i,j=1}^d,\quad \vec{b}:= (b_i)_{i=1}^d,\quad \vec{\beta}:=(\beta_i)_{i=1}^d$$
are used for matrices and vectors.
We also denote $\bS^{d\times d}$ to be the set of $d\times d$ real-valued constant symmetric matrices. The following conditions regarding the elliptic operator and boundary condition will be assumed throughout this paper. All the coefficients in the elliptic operators are assumed to be bounded and measurable, and the leading coefficients are assumed to be symmetric and uniformly elliptic with elliptic constant $\lambda>0$:
\begin{equation*}
\lambda\abs{\zeta}^2 \leq a_{ij}\zeta_i\zeta_j \leq \lambda^{-1}\abs{\zeta}^2,\,\, \forall \zeta\in \bR^{d},\quad \abs{\vec{b}}, \abs{c} \leq K, \quad a_{ij}=a_{ji},
\end{equation*}
where $K>0$ is a constant. We also assume that the boundary operator is oblique, i.e., for some $\delta\in(0,1)$,
\begin{equation}\label{eqn-oblique}
\beta_i n_i \leq -\delta\abs{\vec{\beta}}<0\,\, \quad \text{on }\p\Omega,
\end{equation}
where $\vec{n}=(n_i)_{i=1}^d$ is the unit outward normal direction.

We are interested in better regularity of $W^{2}_2$-strong solutions to \eqref{eqn-main-with-lower-order}. As proved in \cite{DL}, if $\vec{A}$ has locally small bounded mean oscillations (small-BMO), $\beta_0,\vec{\beta}\in C^\alpha,\alpha\in (0,1)$, and $\p\Omega$ can be locally represented by a Lipschitz function with sufficiently small Lipschitz constant, then for any $q\in(1,1/(1-\alpha))$, $f\in L_q(\Omega)$ and $g\in W^{1}_q(\Omega)$ imply that the strong solution $u\in W^{2}_q(\Omega)$. In this paper, we give minimal regularity assumptions on these objects such that any $W^{2}_2(\Omega)$-strong solution is $C^2(\overline{\Omega})$. Due to an example given in \cite[Theorem~4]{EM}, $D^2u$ might not be bounded or even BMO if we merely assume the continuity of $a_{ij}$. Hence certain conditions on its modulus of continuity are needed.

For divergence form equations it is well known that any weak solution is $C^1$ if $a_{ij}$ satisfies the so-called $\alpha$-increasing Dini condition, i.e., the modulus of continuity of $a_{ij}$ is bounded by a Dini function (see definition below) $\phi$ satisfying that $\phi(r)/r^\alpha$ is non-increasing for some $\alpha\in(0,1]$. This is the borderline case of the classical Campanato-type results, see \cite[Section~5]{L87}. In this direction, later Li in \cite{LI} obtained the interior $C^1$-regularity, if
$$
\varphi_{\vec{A}}(r):=\sup_{x} \left(\fint_{B_r(x)}\abs{a_{ij}(y)-a_{ij}(x)}^2\,dy\right)^{1/2}
$$
is a Dini function. Recently, in \cite{DK17} the first named author and Kim generalized this result by assuming that
$$
\omega_{\vec{A}}(r):=\sup_{x} \fint_{B_r(x)}\Big| a_{ij}(y)-\big(\fint_{B_r(x)}a_{ij}\big)\Big|\,dy
$$
is a Dini function, noting $\omega_A(r)\leq 2\varphi_A(r)$. In the same paper, under the same type of regularity assumptions, interior $C^2$ estimate was also proved for equations in non-divergence form. The corresponding boundary estimate for Dirichlet problem in $C^{2,\text{Dini}}$ domain can be found in \cite{DEK}.

Back to the oblique derivative problem, besides the same assumptions on $a_{ij}$ as the interior case, certain regularity assumptions on $\p\Omega$ and the boundary operator are also needed in order to obtain a global result. Direct computation shows that similar to Dirichlet problem, if $\p\Omega \in C^{2,\text{Dini}}$, we can reduce the problem to the half space case by simply flattening the boundary. However, in the same spirit of \cite{Lie,Sa} and \cite{DL}, we expect $\p\Omega$ to be one derivative less regular, i.e., we expect the result still holds when $\p\Omega\in C^{1,\text{Dini}}$. In this regard, a global $C^2$ estimate was proved in \cite{DLK} given that $\p\Omega, \beta_0, \vec{\beta}\in C^{1,\text{Dini}^2}$, which means, e.g., the Dini integrals of their modulus of continuity $I\rho_{\bullet}$ (see definition below) are still Dini functions. The proof was based on an extension idea introduced in \cite{Sa}, which was also used in \cite{DL}. In this paper, using the $W^{2}_p$ estimate in \cite{DL}, the regularized distance, and a delicate decomposition of solutions, we relax the regularity assumption to $\p\Omega, \beta_0, \vec{\beta} \in C^{1,\text{Dini}}$, which seems to be optimal for the global $C^2$ estimate. To the best of our knowledge, such result is new even when the coefficients are smooth.

Now we formulate our problem precisely. A function $\theta: (0,1]\rightarrow [0,\infty)$ is called a Dini function if
$$\int_0^1 \frac{\theta(r)}{r} \, dr<\infty.$$
We write its Dini integral as
$$I\theta(r):=\int_0^r \frac{\theta(t)}{t}\,dt.$$
Both of the following notation are used for the average
$$
(f)_{\Omega} = \fint_{\Omega} f.
$$
We also denote
$$
\Omega_r(x):=\Omega\cap B_r(x),
$$
where the center $x$ will be omitted when it is the origin.
\begin{definition}\label{def-oscillation}
For a function $f$ defined on $\Omega$, we consider two types of its oscillations:
$$\omega_f(r):=\sup_{x\in \overline{\Omega}} \fint_{\Omega_r(x)}\abs{f-(f)_{\Omega_r(x)}},\quad \rho_f(r):=\sup_{|x-y|\leq r, x,y\in \Omega}\abs{f(x)-f(y)}.$$
We say that $f$ is of $L_1$-mean Dini oscillation (uniform Dini) if $\omega_f$ ($\rho_f$, respectively) is a Dini function.
\end{definition}
Clearly, any uniform Dini function is $L_1$-mean Dini. On the other hand, according to a result by Spanne in \cite{S}, on a domain $\Omega$ with exterior measure condition, any $L_1$-mean Dini function $f$ is uniformly continuous with modulus of continuity given by $I\omega_f$. Simple calculation shows that if both $f$ and $g$ are $L_1$-mean Dini functions defined on a bounded domain $\Omega$, then $fg$ is $L_1$-mean Dini.

A function $\psi$ is said to be $C^{1,\text{Dini}}$ if it is continuous with $D\psi$ being uniform Dini. Below we give the formal definition of $C^{1,\text{Dini}}$ domains.
\begin{definition}\label{def-C1Dini-domain}
A bounded domain $\Omega\subset\bR^d$ is said to have $C^{1,\text{Dini}}$ boundary if there exists some $C^{1,\text{Dini}}$ function $\psi_0:\bR^d\rightarrow\bR$, such that
$$\Omega=\set{\psi_0>0}, \quad \abs{D\psi_0}\geq 1\,\, \text{on }\p\Omega.$$
\end{definition}

Now we state our main result of this paper.
\begin{theorem}\label{thm-main}
Consider the problem \eqref{eqn-main-with-lower-order} in a $C^{1,\text{Dini}}$ domain $\Omega$. Assume that $\beta_0,\vec{\beta} \in C^{1,\text{Dini}}$, and $\vec{A}, \vec{b},c$ are of $L_1$-mean Dini oscillation. Let $u\in W^{2}_2(\Omega)$ be a strong solution to \eqref{eqn-main-with-lower-order} with $g \in C^{1,\text{Dini}}$ and $f$ being of $L_1$-mean Dini oscillation. Then $u\in C^2(\overline{\Omega})$.
\end{theorem}

\begin{remark}
Using the $W^2_p$-wellposedness in \cite{DL} and bootstrap, this result still holds if we replace $u \in W^2_2(\Omega)$ with $u\in W^2_q(\Omega)$ for some $q>1$.
\end{remark}
Noting the $W^2_{2}$ solvability in \cite{DL}, we immediately obtain the following.
\begin{corollary}
Besides the assumptions of Theorem \ref{thm-main}, if we further assume
\begin{equation*}
c\geq 0,\quad \beta_0\leq 0,\quad |c|^2+|\beta_0|^2 \not\equiv 0,
\end{equation*}
then there exists a unique $C^2(\overline{\Omega})$ solution to \eqref{eqn-main-with-lower-order}.
\end{corollary}

Our approach is also applicable to concave fully nonlinear elliptic equations $$F[u]:=F(D^2u,Du,u,y)=0$$ under the ``$L_d$-mean Dini'' assumptions given below in (4).
\begin{assumption}\label{ass-nonlinear}
The function $F(M,p,u,y)$ defined on $\bS^{d\times d}\times \bR^d \times \bR\times \Omega$ satisfies
\begin{enumerate}[$(1)$]
\item
$F(\cdot,p,u,y)$ is concave.
\item
There exists a constant $\lambda>0$ such that
$$
\lambda \norm{M} \leq {F(M+N,p,u,y) - F(N,p,u,y)} \leq \lambda^{-1} \norm{M}
$$
for any $N\in\bS^{d\times d}$, $p\in \bR^d$, $u\in \bR$, $y\in \Omega$, and $M\in \bS^{d\times d}, M\geq 0$.
\item
$|F(0,0,0,y)| + [F(M,\cdot,\cdot,y)]_1 \leq K$ for any $M\in \bS^{d\times d}$ and $y\in \Omega$, where $[\ \cdot\ ]_1$ represents the Lipschitz semi-norm.
\item
There exist $R_0>0$ and a Dini function $\omega_F$, such that for any $x\in\overline{\Omega}, r\in(0,R_0]$, we can find some function $F_0:\bS^{d\times d}\times \bR^d \times \bR \rightarrow \bR$ satisfying (1)-(3), such that
\begin{equation*}
\left(\fint_{\Omega_r(x)} \abs{F(M,p,u,y)-F_0(M,p,u)}^d\, dy\right)^{1/d} \leq (|M|+|p|+|u|+1)\omega_F(r).
\end{equation*}
\end{enumerate}
\end{assumption}
It is worth noting that this class includes the Bellman equations
\begin{equation}\label{eqn-Bellman}
\sup_{\omega\in\cA} \{a_{ij}^{\omega}D_{ij} u + b_i^\omega D_iu + c^\omega u - f^\omega\} = 0,
\end{equation}
where all the coefficients as well as $f^\omega$ are uniformly bounded satisfying the $L_d$-mean Dini conditions of the following type:
\begin{equation*}
\left(\fint_{\Omega_r(x)}\sup_\omega\abs{a_{ij}^{\omega}(y)
-(a_{ij}^{\omega})_{\Omega_r(x)}}^d\,dy\right)^{1/d}\lesssim \omega(r).
\end{equation*}
We also impose the sign conditions
\begin{equation}\label{cond-sign}
F(M,p,\cdot,y)\,\,\text{is non-increasing and}\,\,\beta_0\leq-\gamma<0,
\end{equation}
where $\gamma>0$ is a constant.
Now we present the regularity result for fully nonlinear equations.
\begin{theorem}\label{thm-nonlinear}
In a bounded domain $\Omega$, consider the problem $F[u]=0$ with the oblique derivative boundary condition $Bu=g$ on $\partial \Omega$. Under the conditions $\p\Omega, \beta_0,\vec{\beta} \in C^{1,\text{Dini}}$, Assumption \ref{ass-nonlinear}, and \eqref{cond-sign}, for any $C^{1,\text{Dini}}$ boundary data $g$, there exists a unique $C^2(\overline{\Omega})$ solution.
\end{theorem}

Bellman equations \eqref{eqn-Bellman} with oblique derivative boundary condition arise naturally in the study of optimal stochastic control in domain $\Omega$ with reflecting boundary conditions. In \cite{LT}, Lions and Trudinger first studied the $C^{1,1}(\overline{\Omega})\cap C^{2,\alpha}_{loc}(\Omega)$ solvability of Bellman equations, assuming that all the coefficients, $\p\Omega$, and the boundary operators are sufficiently smooth. Later in \cite{Sa}, Safonov proved the unique $C^{2,\alpha}(\overline{\Omega})$ solvability of \eqref{eqn-Bellman} under the relaxed conditions
\begin{equation*}
[a_{ij}^\omega]_{C^\alpha}+[b_i^\omega]_{C^\alpha}+[c^\omega]_{C^\alpha} \leq K,\,\,\forall \omega,\quad \p\Omega, \beta_0, \vec{\beta}\in C^{1,\alpha}.
\end{equation*}
See also \cite{Sa-notes}. Recently, there are also study of similar problems using the viscosity solution approach. For its framework and the solvability, we refer the reader to \cite{I}. In this direction, Milakis and Silvestre in \cite{MS} studied the fully nonlinear, uniformly elliptic equation $F(D^2u)=0$ with Neumann boundary condition on half balls. They showed the $C^\alpha, C^{1,\alpha}$ regularity of viscosity solutions, and the $C^{2,\alpha}$ regularity when $F$ is convex. Later in \cite{LZ}, Li and Zhang proved a similar result for $F(D^2u)=0$ with oblique derivative boundary condition in domain $\Omega$. In particular, when $\p\Omega\in C^{1,\alpha}$ and $F$ is convex, they showed that any viscosity solution is in $C^{2,\alpha}$. The key step there is to prove a boundary Harnack inequality based on an Aleksandrov-Bakel'man-Pucci type estimate, and then design approximating problems. In our paper, due to the usage of the Campanato-type iteration, we are able to deal with the Dini case. Moreover, our operator $F$ are more general depending on lower-order terms and the variable $x$.

\section{Preliminary}
\subsection{Regularized distance and coordinate system}\label{sec-reg-dist}
In Definition \ref{def-C1Dini-domain}, a representation function $\psi_0$ is given for the $C^{1,\text{Dini}}$ domain $\Omega$. Here for studying problems with rough boundaries, we mollify $\psi_0$ properly to obtain a more suitable representation. This is the regularized distance. As in \cite[Lemma~5.1]{DLK}, we can find some function $\psi \in C^{1,\text{Dini}}(\bR^d)\cap C^{\infty}(\Omega)$, satisfying
\begin{equation*}
\Omega=\set{\psi>0},\quad N^{-1}\psi(x) \leq \dist(x,\p\Omega) \leq N\psi(x),\,\, \forall x\in\Omega,
\end{equation*}
\begin{equation}\label{est-moc-psi-moc-psi0}
\abs{D\psi}\leq 1\quad\text{in } \overline{\Omega},\quad \rho_{D\psi}(t)\leq N\rho_{D\psi_0}(ct),
\end{equation}
and for any multi-index $l$ with $\abs{l}=m\geq 2$,
\begin{equation}\label{est-reg-dist-higher-order}
\abs{D^l\psi(x)}\leq C\psi(x)^{1-m}\rho_{D\psi_0}(\psi(x)),\quad \forall x\in\Omega.
\end{equation}
Here $C=C(d,m,\norm{D\psi_0}_{L_\infty(\Omega)})$, and $N, c\in [2,\infty)$ are positive constants depending on $(d,\norm{D\psi_0}_{L_\infty(\Omega)})$.

Next, we introduce the coordinate system adapted to our oblique derivative problem. For any $x_0\in\p\Omega$, we choose an orthonormal coordinate system $y=(y_1,\ldots, y_d)$ centered at $x_0$ such that the $y_d$-axis is in the $\vec{\beta}(x_0)$ direction. Now, noting that $D\psi$ is a inward normal on the boundary, due to the obliqueness \eqref{eqn-oblique}, continuity of $D\psi$, and compactness, we can choose some $R_0 \in (0,1)$ independent of $x_0$, such that \begin{equation}\label{eqn-def-R0}
\frac{\beta(x_0)}{|\beta(x_0)|}D\psi=\frac{\p \psi}{\p y_d}\geq \frac{\delta}{2}\abs{D\psi} \quad \text{in } \Omega_{R_0}(x_0).
\end{equation}
Direct computation shows that there exists some constant $c(\delta,d)>0$, such that for any $r\in(0,R_0],$
\begin{equation}\label{est-nbhd-equiv-ball}
c(\delta,d) < \abs{\Omega_r(x_0)}/\abs{B_r(x_0)}<1.
\end{equation}

\subsection{Estimates on the half space}
We will use the notation
\begin{equation*}
B_r^{+}:=B_r\cap \{x:x_d>0\},\quad \Sigma_r:=\{x:|x|<r, x_d=0\},\quad \sigma_r:=\p B_r\cap \{x:x_d>0\}
\end{equation*}
 throughout this paper.
The first result is a weak type-$(1,1)$ estimate given in \cite[Lemma~2.13]{DLK}.
\begin{lemma}\label{lem-weak-1-1}
Let $\overline{\vec{A}}=(\overline{a_{ij}})$ be a constant symmetric matrix with elliptic constant $\lambda$. Assuming that $f\in L_2$, $u\in W^{2}_2(B_1^{+})$ satisfy
\begin{equation*}
-\overline{a_{ij}}D_{ij}u=f\,\, \text{in}\,B^{+}_1,\quad D_du=0 \,\,\text{on}\,\Sigma_1,\quad u =0\,\,\text{on}\,\sigma_1,
\end{equation*}
then for any $t>0$, we have
\begin{equation}
                                        \label{eq11.09}
\abs{\set{\abs{D^2u}>t}\cap B_1^{+}} \leq \frac{C}{t} \fint_{B_1^{+}}\abs{f},
\end{equation}
where $C=C(d,\lambda)>0$ is a constant.
\end{lemma}
As a corollary, we have the following strong type-$(1,p)$ estimate.
\begin{corollary}\label{cor-strong-p-1}
Under the assumptions of Lemma \ref{lem-weak-1-1}, for any $p\in(0,1)$ we have
\begin{equation*}
\big(\fint_{B_1^{+}} \abs{D^2u}^p\big)^{1/p} \leq C \fint_{B_1^{+}}\abs{f},
\end{equation*}
where $C=C(d,\lambda,p)>0$ is a constant.
\end{corollary}
\begin{proof}
By using \eqref{eq11.09}, for any $T\in (0,\infty)$,
\begin{align*}
\int_{B_1^{+}} \abs{D^2u}^p &= \int_0^\infty p t^{p-1} \abs{\set{\abs{D^2u}>t}\cap B_1^{+}} \,dt\\
&\leq \int_0^T p t^{p-1} \abs{B_1^{+}} \,dt + C\int_T^\infty p t^{p-2} \big(\fint_{B_1^{+}}\abs{f}\big) \,dt\\
&\leq T^p \abs{B_1^{+}} + C\frac{p}{1-p}T^{p-1} \fint_{B_1^{+}}\abs{f}.
\end{align*}
Minimizing in $T$, we obtain the desired estimate.
\end{proof}

We also need the following $C^3$ estimate.
\begin{lemma}\label{lem-C3}
Let $p\in (0,1)$. Assume that $u\in W^2_2(B^{+}_1)$ is a strong solution to
\begin{equation*}
-\overline{a_{ij}}D_{ij}u=\overline{f}\quad \text{in}\,B^{+}_1,\quad D_d u =0\quad\text{on}\,\Sigma_1,
\end{equation*}
where $\overline{a_{ij}}$ is a constant matrix as before and $\overline{f}$ is a constant. Then we have $u\in C^3(\overline{B_{1/2}^{+}})$ satisfying
\begin{equation}\label{est-Lip-Lp}
\norm{D^3 u}_{L_\infty(B_{1/2}^{+})} \leq C(\abs{D^2u - \vq}^p)^{1/p}_{B_1^{+}},
\end{equation}
where $\vq\in \bS^{d\times d}$ can be any constant symmetric matrix and $C=C(d, p, \lambda)>0$.
\end{lemma}
\begin{proof}
First, note that for $x'=(x_1,\ldots,x_{d-1})$, formally $U:=D^2_{x'}u-\vec{q}$ is a solution to
\begin{equation*}
-\overline{a_{ij}}D_{ij}U=0\quad \text{in }B^{+}_1,\quad D_d U =0\quad\text{on}\,\Sigma_1.
\end{equation*}
By using the argument of finite-difference quotients, as a corollary of the Schauder estimate for elliptic equations with the Neumann boundary condition, we have the Lipschitz estimate
\begin{equation*}
\norm{DD^2_{x'}u}_{L_\infty(B_r^{+})}\leq \frac{C}{(R-r)^{1+d/2}}\norm{D^2_{x'}u-\vec{q}}_{L_2(B_R^{+})},
\end{equation*}
for any $0<r<R<1$. From this, we first differentiate the equation in the $x'$ direction to obtain the corresponding estimate for $D_d^2D_{x'}u$. Then we differentiate in the $x_d$ direction for $D^3_du$. Combining these, we have
\begin{equation*}
\norm{D^3u}_{L_\infty(B_r^{+})}\leq \frac{C}{(R-r)^{1+d/2}}\norm{D^2u-\vec{q}}_{L_2(B_R^{+})}.
\end{equation*}
From this we can obtain \eqref{est-Lip-Lp} using interpolation and an iteration argument which can be found, for instance, in \cite[Lemma~2.10]{D-book}.
\end{proof}

\subsection{$L_p$-mean oscillation and more on Dini functions}\label{sec-mosc-minimizer}
In this paper, the following $L_p$-mean oscillation of $D^2u$ will be intensively studied. For any $x\in\overline{\Omega}$ and $p\in(0,1)$,
\begin{equation}\label{eqn-def-Lp-mean-osc}
\phi(x,r) := \inf_{\vq\in \bS^{d\times d}} \big(\fint_{\Omega_r(x)}\abs{D^2u - \vq}^p\big)^{1/p}.
\end{equation}
We note a few properties of such $L_p$-mean oscillation:
\begin{enumerate}
\item If $D^2u\in L_{p,loc}$, for each $x,r$ we can find at least one minimizer. In this paper, we write $\vq_{x,r}$ for such a minimizer.
\item If $D^2u$ is continuous at $x$, then $\vq_{x,r}\rightarrow D^2u(x)$ as $r\rightarrow \infty$. Indeed,
\begin{equation*}
\begin{split}
\abs{D^2u(x)-\vq_{x,r}}^p &\leq \fint_{\Omega_r(x)}\abs{D^2u(z)-\vq_{x,r}}^p\,dz + \fint_{\Omega_r(x)}\abs{D^2u(z)-D^2u(x)}^p\,dz\\
&\leq 2\fint_{\Omega_r(x)}\abs{D^2u(z)-D^2u(x)}^p\,dz\rightarrow 0\,\,\text{as }r\rightarrow 0.
\end{split}
\end{equation*}
Here in the last inequality, we used the fact that $\vq_{x,r}$ is a minimizer.
\end{enumerate}
Later in the proof we will see, many steps in the classical Campanato's iteration for $L_2$-mean oscillation still work if we replace $(D^2u)_{\Omega_r(x)}$ with $\vq_{r,x}$.

The following property of Dini functions is useful in our iteration argument. This iteration was introduced in \cite{D}, and was also used in \cite{DK17,DEK,DLK} for studying equations with $L_1$-mean Dini coefficients. A nonnegative function $\omega$ is called comparable if there exists some constant $\kappa\in(0,1)$, such that
\begin{equation}\label{eqn-comparability}
\omega(s)/\omega(t)\leq C(\kappa),\quad \forall s,t\in(0,1]\,\,\text{and}\,s/t\in[\kappa,\kappa^{-1}].
\end{equation}
\begin{lemma}\label{lem-induced-Dini-function}
Assume that $\omega$ is a Dini function satisfying the almost increasing condition
\begin{equation}
                        \label{eq4.58}
\omega(r)\ge C\omega(s),\quad\forall\ r/2<s<r<\infty
\end{equation}
for some constant $C>0$. Then
\begin{equation}\label{eqn-derived-Dini-function}
\widetilde{\omega}(r):=\sum_{i=0}^\infty (1/2)^{i}\big(\omega(\kappa^{-i}r)1_{\kappa^{-i}r<1} + \omega(1)1_{\kappa^{-i}r\geq 1}\big)
\end{equation}
is also a Dini function satisfying \eqref{eqn-comparability}.
Furthermore, up to a constant depending only on $\kappa$, we have
\begin{equation*}
I\widetilde{\omega}(r)\simeq \sum_{j=0}^{\infty}\widetilde{\omega}(\kappa^jr).
\end{equation*}
In particular, $\widetilde{\omega}(r)\rightarrow 0$ as $r\rightarrow 0$.
\end{lemma}
The proof is by direct computation which can be found in \cite[Lemma~1]{D} and \cite[Lemma~2.7]{DK17}.

Clearly, $\rho_f$ given in Definition \ref{def-oscillation} is a non-decreasing function. One can simply check that $\omega_f$ satisfies \eqref{eq4.58} provided that the doubling property is satisfied, i.e., for any $x\in \overline{\Omega}$ and $r>0$, $|\Omega_{2r}(x)|\le C_0 |\Omega_{r}(x)|$ for some constant $C_0>0$.

\section{$L_p$-Mean Oscillation Estimate}
In this section, we focus on the equation without lower-order terms
\begin{equation}\label{eqn-no-lower-order}
\begin{cases}
a_{ij}D_{ij}u = f &\text{in}\,\Omega,\\
\beta_i D_i u = g &\text{on}\,\p\Omega.
\end{cases}
\end{equation}
For any $x\in\overline{\Omega}$ and $p\in(0,1)$, recall the $L_p$-mean oscillation of $D^2u$ in \eqref{eqn-def-Lp-mean-osc}. For simplicity, we also denote
\begin{align*}
&\omega^{(1)}(r):= r\norm{D\beta}_{L_\infty(\Omega)} + \rho_{D\vec{\beta}}(r) + \rho_{D\psi_0}(r) + \omega_{\vec{A}}(r),\\
&\omega^{(2)}(r):= r\norm{Dg}_{L_\infty(\Omega)} + \rho_{Dg}(r) + \omega_f(r).
\end{align*}
Due to our assumptions, they are both Dini functions. The next proposition plays a key role in proving Theorem \ref{thm-main}.
\begin{proposition}\label{prop-mean-oscillation}
Assume that $u\in C^2(\overline{\Omega})$ solves \eqref{eqn-no-lower-order} in a bounded domain $\Omega$. If $\vec{A}$ and $f$ are $L_1$-mean Dini, and $\p\Omega,\vec{\beta},g\in C^{1,\text{Dini}}$, then for any $x\in \overline{\Omega}, p\in (0,1),\kappa \in (0,1)$, and $r\leq R_0$, we have
\begin{align}\label{est-u-mean-Dini}
&\phi(x,\kappa r)\leq C\kappa \phi(x,r) \nonumber\\ &\quad +C(\kappa^{-d/p}+\kappa)\big((\norm{D^2u}_{L_\infty(\Omega_r(x))} + \norm{Du}_{L_\infty(\Omega_r(x))}+\norm{Dg}_{L_\infty(\Omega)})\omega^{(1)}(r) + \omega^{(2)}(r)\big),
\end{align}
where $C=C(d,p,\lambda,\delta,\norm{\beta}_{C^1(\Omega)}, \norm{D\psi_0}_{L_\infty(\Omega)},\rho_{D\psi_0})$.
\end{proposition}
The rest of this section will be devoted to its proof.
\subsection{Homogeneous case}
We first consider the equation with constant coefficients and homogeneous boundary condition. Let $(\overline{a_{ij}})$ be a constant symmetric matrix with elliptic constant $\lambda$.  For $x\in\p\Omega$, we choose the coordinate system $y=(y_1,\ldots,y_d)$ centered at $x$ as before. Recall that in $\Omega_{R_0}$, we have
$$
\frac{\p \psi}{\p y_d}\geq \delta\abs{D\psi}/2,
$$
i.e., the $y_d$-direction is oblique. For $r\leq R_0$, denote $\Gamma_r:=\p\Omega\cap B_r$.
\begin{lemma}\label{lem-homogeneous}
Assume that $v\in W^2_2(\Omega_r)$ and $f \in L_2(\Omega_r)$ satisfy
\begin{equation}\label{eqn-homogeneous}
\begin{cases}
-\overline{a_{ij}}D_{ij}v = f &\text{in}\,\Omega_r,\\
\frac{\p v}{\p y_d}=0 &\text{on}\,\Gamma_r.
\end{cases}
\end{equation}
Then for each $\kappa \in (0,1)$, we can find some constant matrix $\vec{V}_{\kappa r}\in \bS^{d\times d}$ such that
\begin{equation}\label{est-mean-osc-homogeneous}
\begin{split}
\big(\fint_{\Omega_{\kappa r}} \abs{D^2 v - \vec{V}_{\kappa r}}^p\big)^{1/p} \leq& C \kappa \big(\fint_{\Omega_r} \abs{D^2 v - \vq}^p\big)^{1/p}\\
&+ C(\kappa^{-d/p}+\kappa)\big((\abs{f-(f)_{\Omega_r}})_{\Omega_r} + \rho_{D\psi_0}(r)(\abs{D^2v}^2)_{\Omega_r}^{1/2}\big)
\end{split}
\end{equation}
holds for any $\vq \in \bS^{d\times d}$,
where $C=C(d,p,\lambda,\delta,\norm{D\psi_0}_{L_\infty(\Omega)})$ is a constant.
\end{lemma}
\begin{proof}
Clearly we only need to prove \eqref{est-mean-osc-homogeneous} for small $\kappa$. Here we consider $\kappa<\delta/(4M)$ where $M:=\max\set{c\delta, 2c/\inf_{\p\Omega}\abs{D\psi}}$ with $c$ given as in \eqref{est-moc-psi-moc-psi0}.

In $\Omega_r$, we flatten the boundary by taking the change of variables
\begin{equation}\label{eqn-cov}
y\in \Omega_r\cup\Gamma_r \rightarrow z\in \overline{\bR^d_{+}},\quad z_i(y) = y_i,\,\,i<d,\quad z_d(y)=\psi(y).
\end{equation}
In the $z$-variables, \eqref{eqn-homogeneous} becomes
\begin{equation*}
\begin{cases}
-\overline{a_{ij}}\frac{\p z_k}{\p y_i}\frac{\p z_l}{\p y_j}\frac{\p^2}{\p z_k \p z_l}v = f + \overline{a_{ij}}\frac{\p^2 \psi}{\p y_i \p y_j}\frac{\p v}{\p z_d} &\text{in}\,\,z(\Omega_r)\subset\bR^d_{+},\\
\frac{\p v}{\p z_d}=0 &\text{on}\,\,z(\Gamma_r)\subset\set{z_d=0}.
\end{cases}
\end{equation*}
In the sequel, we denote
$$\widetilde{\vec{A}} = (\widetilde{a_{kl}})_{k,l}:=\big(\overline{a_{ij}}\frac{\p z_k}{\p y_i}\frac{\p z_l}{\p y_j}\big)_{k,l}.$$
Due to \eqref{est-moc-psi-moc-psi0} and our choice of $M$, $\widetilde{\vec{A}}$ is uniform Dini with
\begin{equation}\label{est-moc-A-tilde}
\rho_{\widetilde{\vec{A}}}\big(\frac{\delta}{M}r\big) \leq C(d, \lambda,\norm{D\psi}_{L_\infty(\Omega)})\rho_{D\psi_0}(r).
\end{equation}
Since $\frac{\p \psi}{\p y_d}\geq \frac{\delta}{2}\abs{D\psi}$ in $\Omega_r$, for any two points $z^{(1)}, z^{(2)} \in z(\Omega_r)$ we have
\begin{align*}
\abs{z^{(1)}-z^{(2)}}^2 &= \sum_{i<d}\abs{y^{(1)}_i-y^{(2)}_i}^2 + \abs{\psi(y^{(1)})-\psi(y^{(2)})}^2\\
&\geq \sum_{i<d}\abs{y^{(1)}_i-y^{(2)}_i}^2 + \big(\frac{\delta}{2}\abs{D\psi}\big)^2\abs{y^{(1)}_d - y^{(2)}_d}^2\\
&\geq \big(\frac{\delta}{2}\abs{D\psi}\big)^2\abs{y^{(1)}-y^{(2)}}^2 \geq \big(\frac{\delta c}{M}\big)^2\abs{y^{(1)}-y^{(2)}}^2,
\end{align*}
and similarly
\begin{equation*}
\abs{z^{(1)}-z^{(2)}}^2 \leq 2 \abs{y^{(1)}-y^{(2)}}^2.
\end{equation*}
From these and $\kappa<\delta/(4M)$, we know
\begin{equation}\label{eqn-inclusion}
z(\Omega_{\kappa r}) \subset B^{+}_{2 \kappa r} \subset B^{+}_{(\delta/2M)r} \subset B^{+}_{(\delta/M)r} \subset z(\Omega_{r/c})\subset z(\Omega_r).
\end{equation}
In $B^{+}_{(\delta/M)r}$, we decompose $v=\theta+\xi$, where $\theta\in W^2_2(B^{+}_{(\delta/M)r})$ is a strong solution to
\begin{equation}\label{eqn-theta}
\begin{cases}
-(\widetilde{a_{kl}})_{B^{+}_{(\delta/M)r}}\frac{\p^2}{\p z_k \p z_l}\theta&= (\widetilde{a_{kl}}- (\widetilde{a_{kl}})_{B^{+}_{(\delta/M)r}})\frac{\p^2}{\p z_k \p z_l}v + \overline{a_{ij}}\frac{\p^2 \psi}{\p y_i \p y_j}\frac{\p v}{\p z_d} + f-(f)_{\Omega_r} \,\,\text{in}\,B^{+}_{(\delta/M)r},\\
\frac{\p \theta}{\p z_d} &= 0 \,\,\text{on}\,\Sigma_{(\delta/M)r},\\
\theta &= 0 \,\,\text{on}\,\sigma_{(\delta/M)r}.
\end{cases}
\end{equation}
Recall that the notation $(f)_{\Omega_r}$ stands for the average. Due to \eqref{est-reg-dist-higher-order}, $\frac{\p v}{\p z_d}=0$ on $z_d=0$, and Hardy's inequality, we have
\begin{equation*}
\Abs{\overline{a_{ij}}\frac{\p^2 \psi}{\p y_i \p y_j}\frac{\p v}{\p z_d}}\leq C\rho_{D\psi_0}(r)\frac{1}{z_d}\Abs{\frac{\p v}{\p z_d}}\in L_2(B_{(\delta/M)r}^{+}),
\end{equation*}
with
\begin{equation}\label{est-use-hardy}
\Norm{\overline{a_{ij}}\frac{\p^2 \psi}{\p y_i \p y_j}\frac{\p v}{\p z_d}}_{L_2(B_{(\delta/M)r}^{+})} \leq C\rho_{D\psi_0}(r) \norm{D^2_{z}v}_{L_2(B_{(\delta/M)r}^{+})}.
\end{equation}
Such solution $\theta \in W^2_2(B_{(\delta/M)r}^{+})$ exists. Indeed, we first reduce \eqref{eqn-theta} to a Dirichlet problem in $B_{(\delta/M)r}$ by taking the even extension in $z_d$ for $\theta$ and the source term, and the following extension for the leading coefficients
\begin{equation*}
\hat{a}_{kl}=\begin{cases}
(\widetilde{a_{kl}})_{B^{+}_{(\delta/M)r}}\,\,z_d\geq 0,\\
\epsi_k\epsi_l(\widetilde{a_{kl}})_{B^{+}_{(\delta/M)r}}\,\,z_d<0,
\end{cases}\,\,\text{where}\, \epsi_k=\begin{cases}
+1\,\,\text{when}\,\,k\neq d,\\
-1\,\,\text{when}\,\,k=d.
\end{cases}
\end{equation*}
Note that the extended problem has measurable coefficients only depending on $z_d$, which are also continuous (actually equal to constants) near $z_d=\pm 1$. Then the $W^2_2$ solvability follows from \cite[Theorem~2.8]{DK10}. See also the example on \cite[pp.~6483]{DK10}.

Using Corollary \ref{cor-strong-p-1}, \eqref{est-use-hardy}, and H\"older's inequality, we obtain
\begin{equation}\label{est-theta}
\begin{split}
(\abs{D^2\theta}^p)_{B^{+}_{(\delta/M)r}}^{1/p} &\lesssim \rho_{\widetilde{\vec{A}}}\big(\frac{\delta}{M}r\big)(\abs{D^2_z v}^2)_{B^{+}_{(\delta/M)r}}^{1/2} + \rho_{D\psi_0}(r) (\abs{D^2_{z}v}^2)_{B_{(\delta/M)r}^{+}}^{1/2}+ (\abs{f-(f)_{\Omega_r}})_{B^{+}_{(\delta/M)r}}\\
&\lesssim \rho_{D\psi_0}(r) (\abs{D^2_{z}v}^2)_{B_{(\delta/M)r}^{+}}^{1/2}+ (\abs{f-(f)_{\Omega_r}})_{B^{+}_{(\delta/M)r}},
\end{split}
\end{equation}
where in the last inequality, we used \eqref{est-moc-A-tilde}. Here the implicit constant depends on $d$, $p$, $\lambda$, and $\norm{D\psi_0}_{L_\infty(\Omega)}$.

Now $\xi: = v-\theta \in W^2_2(B_{(\delta/M)r}^{+})$ satisfies
\begin{equation*}
\begin{cases}
-(\widetilde{a_{kl}})_{B^{+}_{(\delta/M)r}}\frac{\p^2}{\p z_k \p z_l}\xi = (f)_{\Omega_r} &\text{in}\,B^{+}_{(\delta/M)r},\\
\frac{\p\xi}{\p z_d}=0 &\text{on}\,\Sigma_{(\delta/M)r}.
\end{cases}
\end{equation*}
Noting \eqref{eqn-inclusion}, a rescaled version of Lemma \ref{lem-C3} leads to
\begin{equation*}
\begin{split}
&\big(\fint_{B_{2\kappa r}^{+}} \abs{D^2\xi - (D^2\xi)_{B_{2\kappa r}^{+}}}^p\big)^{1/p} \leq C(d)\kappa r \norm{D^3\xi}_{L_\infty(B_{2\kappa r}^{+})} \\
&\leq C(d)\kappa r \norm{D^3\xi}_{L_\infty(B_{(\delta/2M) r}^{+})}
\leq C(d,p,\lambda)\kappa (\abs{D^2\xi - \widetilde{\vec{q}}}^p)^{1/p}_{B_{(\delta/M)r}^{+}}
\end{split}
\end{equation*}
with $\widetilde{\vec{q}}\in\bS^{d\times d}$ to be chosen later.

Combining this and \eqref{est-theta}, we obtain
\begin{align}
&\big(\fint_{B_{2\kappa r}^{+}} \abs{D^2_zv - (D^2\xi)_{B^+_{2\kappa r}}}^p\big)^{1/p}\nonumber\\
&\leq 2^{1/p-1}\big(\big(\fint_{B_{2\kappa r}^{+}} \abs{D^2 \xi - (D^2\xi)_{B^+_{2\kappa r}}}^p\big)^{1/p} + \big(\fint_{B_{2\kappa r}^{+}} \abs{D^2_z\theta}^p\big)^{1/p}\big)\nonumber\\
&\lesssim \kappa (\abs{D^2\xi - \widetilde{\vq}}^p)^{1/p}_{B_{(\delta/M)r}^{+}} + \big(\fint_{B_{2\kappa r}^{+}} \abs{D^2_z\theta}^p\big)^{1/p}\nonumber\\
&\lesssim \kappa (\abs{D^2_z v - \widetilde{\vq}}^p)^{1/p}_{B_{(\delta/M)r}^{+}} + (\kappa^{-d/p}+\kappa)\big(\fint_{B_{(\delta/M) r}^{+}} \abs{D^2_z\theta}^p\big)^{1/p}\nonumber\\
&\lesssim \kappa (\abs{D^2_z v - \widetilde{\vq}}^p)^{1/p}_{B_{(\delta/M)r}^{+}} + (\kappa^{-d/p}+\kappa)\big(\rho_{D\psi_0}(r) (\abs{D^2_{z}v}^2)_{B_{(\delta/M)r}^{+}}^{1/2}+ (\abs{f-(f)_{\Omega_r}})_{B^{+}_{(\delta/M)r}}\big)\label{eqn-03292024}
\end{align}
with the constant depending on $(d,p,\lambda,\norm{D\psi_0}_{L_\infty(\Omega)})$. Now we translate back to the $y$-coordinates. Combining
\begin{align*}
&dz = \Abs{\frac{\p\psi}{\p y_d}}dy,\quad \frac{\delta}{2}\abs{D\psi}\leq \Abs{\frac{\p\psi}{\p y_d}} \leq 1\,\, \forall y\in\Omega_{r}, \\
&\frac{\p^2 v}{\p y_i \p y_j} = \frac{\p^2 v}{\p z_k \p z_l}\frac{\p z_k}{\p y_i}\frac{\p z_l}{\p y_j} + \frac{\p v}{\p z_d}\frac{\p^2 \psi}{\p y_i \p y_j}
\end{align*}
together with \eqref{eqn-inclusion}, Hardy's inequality, and H\"older's inequality, we can continue the computation from \eqref{eqn-03292024}:
\begin{equation}\label{est-v-1}
\begin{split}
&\big(\fint_{\Omega_{\kappa r}} \abs{(\frac{\p y}{\p z})^T D^2_y v \frac{\p y}{\p z}  - (D^2\xi)_{B^+_{2\kappa r}}}^p\,dy\big)^{1/p}\\
&\lesssim \big(\fint_{B_{2\kappa r}^{+}} \abs{D^2_zv - (D^2\xi)_{B^+_{2\kappa r}}}^p\,dz\big)^{1/p} + \rho_{D\psi_0}(r) (\abs{D^2_{z}v}^2)_{B_{2\kappa r}^{+}}^{1/2}\\
&\lesssim \kappa (\abs{D^2_z v - \widetilde{\vq}}^p)^{1/p}_{B_{(\delta/M)r}^{+}} + (\kappa^{-d/p}+\kappa)\big(\rho_{D\psi_0}(r) (\abs{D^2_{z}v}^2)_{B_{(\delta/M)r}^{+}}^{1/2}+ (\abs{f-(f)_{\Omega_r}})_{B^{+}_{(\delta/M)r}}\big)\\
&\lesssim \kappa\big(\fint_{\Omega_{r/c}} \abs{(\frac{\p y}{\p z})^T D^2_y v \frac{\p y}{\p z} - \widetilde{\vq}}^p\,dy\big)^{1/p} + (\kappa^{-d/p}+\kappa)\big(\rho_{D\psi_0}(r) (\abs{D^2_{y}v}^2)_{\Omega_r}^{1/2} + (\abs{f-(f)_{\Omega_r}})_{\Omega_r}\big).
\end{split}
\end{equation}
Here besides $(d,p,\lambda,\norm{D\psi_0}_{L_\infty(\Omega)})$, the constant also depends on $\delta$. This is almost \eqref{est-mean-osc-homogeneous}, except that we also need to deal with $\p z/\p y$ coming from the change of variables. By \eqref{est-moc-psi-moc-psi0} and the (generalized) triangular inequality,
\begin{equation*}
\begin{split}
&\big(\fint_{\Omega_{r/c}} \abs{(\frac{\p y}{\p z})^T D^2_y v \frac{\p y}{\p z} - \widetilde{\vq}}^p\,dy\big)^{1/p} \\
&\leq 3^{1/p-1}\big(2 \rho_{D\psi}(r/c) \norm{\frac{\p y}{\p z}}_{L_\infty(\Omega_{r/c})}(\abs{D^2_y v}^2)^{1/2}_{\Omega_{r/c}}\\
&\quad + \norm{\frac{\p y}{\p z}}^2_{L_\infty(\Omega_{r/c})}\big(\fint_{\Omega_{r/c}} \Abs{D^2_y v - \big(\frac{\p z}{\p y}(0)\big)^T \widetilde{\vq} \frac{\p z}{\p y}(0)
}^p\,dy\big)^{1/p}\big)\\
&\leq C(d,p,\delta,\norm{D\psi_0}_{L_\infty(\Omega)}) \big(\rho_{D\psi_0}(r)(\abs{D^2_y v}^2)^{1/2}_{\Omega_r} + \big(\fint_{\Omega_{r}} \Abs{D^2_y v - \vq}^p\,dy\big)^{1/p}\big),
\end{split}
\end{equation*}
if we take
\begin{equation*}
\widetilde{\vq}=\big(\frac{\p y}{\p z}(0)\big)^T \vq \frac{\p y}{\p z}(0).
\end{equation*}

If we also take \eqref{eqn-def-R0} into consideration, similar computation leads to
\begin{equation}\label{est-v-3}
\begin{split}
&\big(\fint_{\Omega_{\kappa r}} \abs{\big(\frac{\p y}{\p z}\big)^T D^2_y v \frac{\p y}{\p z} - (D^2\xi)_{B^+_{2\kappa r}}}^p\,dy\big)^{1/p} \\
&\geq C(d,p,\delta,\norm{D\psi}_{L_\infty(\Omega)}) \Big(\big(\fint_{\Omega_{\kappa r}} \Abs{D^2_y v - \vec{V}_{\kappa r}}^p\,dy\big)^{1/p} -\kappa^{-d/2}\rho_{D\psi_0}(r)(\abs{D^2_y v}^2)^{1/2}_{\Omega_r}\Big),
\end{split}
\end{equation}
where
\begin{equation*}
\vec{V}_{\kappa r}=\big(\frac{\p z}{\p y}(0)\big)^T(D^2\xi)_{B^+_{2\kappa r}}\frac{\p z}{\p y}(0).
\end{equation*}
Now combining \eqref{est-v-1}-\eqref{est-v-3}, we immediately obtain \eqref{est-mean-osc-homogeneous}.
\end{proof}

\subsection{Mean oscillation estimate for $D^2u$}
\begin{proof}[Proof of Proposition \ref{prop-mean-oscillation}]
First, for $x$ satisfying $B_r(x)\subset\Omega$, in which case only the interior estimates are concerned, the decay of $L_p$-mean oscillation can be found in \cite[pp.~427]{DK17}. Actually we have
\begin{equation*}
\phi(x,\kappa r)\leq N\kappa \phi(x,r) + N(\kappa^{-d/p}+\kappa)(\norm{D^2u}_{L_\infty(\Omega_r(x))} + \omega_f(r)).
\end{equation*}
By a standard argument, it suffices to consider the case when $x\in\p\Omega$. Choose the coordinate system $y$ centered at $x$ as in Section \ref{sec-reg-dist}. We now reduce the original problem to the homogeneous case \eqref{eqn-homogeneous}. For this, we introduce two auxiliary functions. Let $w\in W^2_2(\Omega)$ be the strong solution to
\begin{equation}\label{eqn-def-w}
\begin{cases}
-\Delta w + w = 0 &\text{in}\,\Omega,\\
\eta\vec\beta(0)\cdot Dw + (1-\eta)\vec\beta\cdot Dw = -\overline{\vec{\beta}}\cdot Du - \eta(y\cdot D\vec{\beta}(0))(Du - Du(0)) + \overline{g} &\text{on}\,\p\Omega.
\end{cases}
\end{equation}
where
\begin{equation*}
\overline{\vec{\beta}}:=\eta(\vec{\beta} - \vec{\beta}(0)- (y\cdot D)\vec{\beta}(0)),\quad \overline{g}:= \eta (g- g(0)-(y\cdot D)g(0)).
\end{equation*}
In the above, $\eta\in C^\infty_c(B_r)$ is taken to be a usual cut-off function satisfying $\eta=1$ on $B_{r/2}$ and $\abs{D\eta}\lesssim1/r$. Due to our previous choice of $R_0$, we know the boundary condition is uniform oblique because
\begin{equation*}
\eta \vec{\beta}(0)\cdot D\psi + (1-\eta)\vec{\beta}\cdot D\psi \geq \eta\frac{\delta}{M}|\vec\beta(0)| + (1-\eta)\delta\abs{\vec{\beta}}\geq \delta \min\set{|\vec\beta(0)|/M,\abs{\vec{\beta}}}.
\end{equation*}
According to \cite[Theorem~2.4]{DL}, such $w$ exists and it satisfies
\begin{equation*}
\norm{w}_{W^2_2(\Omega)} \leq C\big(\norm{\overline{\vec{\beta}}\cdot Du}_{W^1_2(\Omega)} + \norm{\eta(y\cdot D\vec{\beta}(0))(Du - Du(0))}_{W^1_2(\Omega)} + \norm{\overline{g}}_{W^1_2(\Omega)}\big).
\end{equation*}
From this, we obtain
\begin{align}\label{est-w}
&(\abs{D^2w}^2)^{1/2}_{\Omega_{r/2}} \nonumber\\
&\leq C\big((r\norm{D\beta}_{L_\infty}+\rho_{D\beta}(r))(\norm{D^2u}_{L_\infty(\Omega_r)} + \norm{Du}_{L_\infty(\Omega_r)}) + r\norm{Dg}_{L_\infty} + \rho_{Dg}(r)\big),
\end{align}
where $C$ is a constant depending on $(d,\delta,\norm{\beta}_{C^1}, \norm{D\psi_0}_{L_\infty(\Omega)},\rho_{D\psi_0})$.

Next, we consider the parabola
\begin{equation}\label{eqn-def-h}
h:=-\big(y\cdot D\vec{\beta}(0)\big)Du(0)y_d  + g(0)y_d + y\cdot Dg(0) y_d + \frac{y_d^2 (D_d\vec{\beta}(0))\cdot Du(0) -  y_d^2 D_d g(0)}{2},
\end{equation}
which satisfies
\begin{equation*}
\frac{\p h}{\p y_d}= -\big(y\cdot D\vec{\beta}(0)\big) Du(0) + g(0) + y\cdot Dg(0).
\end{equation*}
From our construction, $v:=u-w-h$ satisfies
\begin{equation*}
\begin{cases}
-(a_{ij})_{\Omega_{r/2}}D_{ij}v = \widetilde{f}:= -((a_{ij})_{\Omega_{r/2}}-a_{ij})D_{ij}u + f + (a_{ij})_{\Omega_{r/2}}(D_{ij}h + D_{ij}w) \quad \text{in }\Omega_{r/2},\\
\frac{\p v}{\p y_d}=0\quad \text{on }\Gamma_{r/2}.
\end{cases}
\end{equation*}
Clearly $\widetilde{f}\in L_2(\Omega_{r/2})$. By the triangular inequality and H\"older's inequality,
\begin{equation}\label{est-f-tilde}
(\abs{\widetilde{f}-(\widetilde{f})_{\Omega_{r/2}}})_{\Omega_{r/2}}
\lesssim \omega_{\vec{A}}(r/2) \norm{D^2u}_{L_\infty(\Omega_{r/2})} + \omega_f(r/2) + (\abs{D^2w}^2)^{1/2}_{\Omega_{r/2}}.
\end{equation}
Now we apply Lemma \ref{lem-homogeneous} with $\kappa$, $r$, and $f$ replaced by $2\kappa$, $r/2$, and $\widetilde{f}$ to obtain that there exists some $\vec{V}_{\kappa r}\in \bS^{d\times d}$ such that for any $\vq\in \bS^{d\times d}$,
\begin{equation*}
\begin{split}
\big(\fint_{\Omega_{\kappa r}} \abs{D^2v - \vec{V}_{\kappa r}}^p\big)^{1/p} \leq& C\big(\kappa \big(\fint_{\Omega_{r/2}}\abs{D^2v - \vq}^p\big)^{1/p}\\
&+ (\kappa^{-d/p}+\kappa)((\abs{\widetilde{f}-(\widetilde{f})_{\Omega_{r/2}}})_{\Omega_{r/2}} + \rho_{D\psi_0}(r/2)(\abs{D^2v}^2)^{1/2}_{\Omega_{r/2}})\big).
\end{split}
\end{equation*}
Noting $u=v+w+h$, we can further estimate the mean oscillation of $D^2u$ by
\begin{equation}\label{est-u-mean-osc-mid}
\begin{split}
&\big(\fint_{\Omega_{\kappa r}} \abs{D^2u - (\vec{V}_{\kappa r}+D^2h)}^p\big)^{1/p}\\
&\leq C\kappa \big(\fint_{\Omega_{r/2}}\abs{D^2u - (\vq + D^2h)}^p\big)^{1/p} + C(\kappa^{-d/p}+\kappa)\big((\abs{\widetilde{f}-(\widetilde{f})_{\Omega_{r/2}}})_{\Omega_{r/2}}
+ (\abs{D^2w}^2)_{\Omega_{r/2}}^{1/2} \\
&\quad+ \rho_{D\psi_0}(r/2)(\norm{D^2u}_{L_\infty(\Omega_{r/2})}+\norm{D\vec{\beta}}_{L_\infty(\Omega)}\norm{Du}_{L_\infty(\Omega_{r/2})} +\norm{Dg}_{L_\infty(\Omega)})\big)
\end{split}
\end{equation}
by applying the (generalized) triangular inequality and H\"older's inequality.
Here, we also used the fact that $D^2h\in \bS^{d\times d}$ is a constant matrix and the inequality
\begin{equation*}
\abs{D^2h} \leq C\norm{D\vec{\beta}}_{L_\infty(\Omega)}\norm{Du}_{L_\infty(\Omega_{r/2})} + \norm{Dg}_{L_\infty(\Omega)}.
\end{equation*}
Now substituting the $w$ and $\widetilde{f}$ terms in \eqref{est-u-mean-osc-mid} by the corresponding estimates \eqref{est-w} and \eqref{est-f-tilde}, taking infimum in $\vq$, we obtain \eqref{est-u-mean-Dini}. The proposition is proved.
\end{proof}
\section{Proof of Theorem \ref{thm-main}}
With all these preparations, we are ready to give the proof of our main results. To begin with, we make some reductions. Rewrite \eqref{eqn-main-with-lower-order} as
\begin{equation*}
\begin{cases}
-a_{ij}D_{ij}u + u = f - b_i D_i u - (c-1)u &\text{in}\,\Omega,\\
\beta_iD_i u = g - \beta_0 u &\text{on}\,\p\Omega.
\end{cases}
\end{equation*}
From \cite{S}, $\int_0^1\omega_f(r)/r\,dr<\infty$ implies that $f$ is uniformly continuous, hence, is bounded. Also, by the Sobolev embedding, we have
$$
b_iD_i u + (c-1)u,\,\beta_0 u \in L_{2^{*}}(\Omega),
$$
where
\begin{equation*}
2^{*} = \begin{cases}
2d/(d-2) &\text{when}\,d>2,\\
d+1 &\text{when}\,d\leq 2.
\end{cases}
\end{equation*}
Using the uniqueness of $W^2_p$ solutions in \cite{DL}, we have $u \in W^2_{2^{*}}$. Repeating if needed, in finite steps we obtain that
$$u\in  W^2_{d+1} \subset C^{1,1/(d+1)}.$$
Since the coefficients $\vec{b}$ and $c$ have $L_1$-mean Dini oscillations, and $\beta_0 \in C^{1,\text{Dini}}$, we can deduce that $b_iD_i u$ and $cu$ are of $L_1$-mean Dini, and $\beta_0 u \in C^{1,\text{Dini}}$. Now, by moving all the lower-order terms to the right-hand side we only need to consider the equation \eqref{eqn-no-lower-order}. Also, due to the approximation given at the end of this section, we only need to prove an a priori estimate. In other words, we will estimate the modulus of continuity of $D^2u$ assuming that $u \in C^2(\overline{\Omega})$. Under all these reductions, Proposition \ref{prop-mean-oscillation} applies.

First we will derive the decay of $L_p$-mean oscillation $\phi(x,r)$ from \eqref{est-u-mean-Dini}. From now on, we fix some $p\in(0,1)$ and then choose $\kappa\in (0,1)$ small enough such that $C\kappa <1/2$ in \eqref{est-u-mean-Dini} to get, for any $r\leq R_0$,
\begin{equation*}
\phi(x,\kappa r)
\leq \frac{1}{2}\phi(x,r) + C(\kappa)\big((\norm{D^2u}_{L_\infty(\Omega)} + \norm{Du}_{L_\infty(\Omega)}+\norm{Dg}_{L_\infty(\Omega)})\omega^{(1)}(r) + \omega^{(2)}(r)\big).
\end{equation*}
Applying this $j$ times, we have
\begin{align}\label{est-mean-osc-j-decay}
&\phi(x,\kappa^j r) \leq (1/2)^j \phi(x,r) \nonumber\\
&\quad+ C\big((\norm{D^2u}_{L_\infty(\Omega)} + \norm{Du}_{L_\infty(\Omega)}+\norm{Dg}_{L_\infty(\Omega)})\widetilde{\omega}^{(1)}(\kappa^{j-1}r) + \widetilde{\omega}^{(2)}(\kappa^{j-1}r)\big),
\end{align}
where $\widetilde{\omega}^{1}$ and $\widetilde{\omega}^{(2)}$ are Dini functions derived from $\omega^{(1)}$ and $\omega^{(2)}$ as in \eqref{eqn-derived-Dini-function}.

Second, we estimate $\norm{D^2u}_{L_\infty(\Omega)}$. For any point $x\in\overline{\Omega}$, we take $\vq_{x, r}$ as a minimizer in $\phi(x,r)$, which exists as explained in Section \ref{sec-mosc-minimizer}. By the triangular inequality,
\begin{equation*}
\abs{\vq_{x,\kappa^{j+1} r} - \vq_{x, \kappa^j r}}^p \leq \abs{\vq_{x,\kappa^{j+1} r} - D^2u(z)}^p + \abs{D^2u(z) - \vq_{x,\kappa^{j} r}}^p
\end{equation*}
holds for any integer $j\geq 0$.
Taking the average for $z\in \Omega_{\kappa^{j+1} r}(x)$, we obtain
\begin{equation*}
\abs{\vq_{x,\kappa^{j+1} r} - \vq_{x,\kappa^j r}} \leq 2^{1/p-1} \big(\phi(x,\kappa^{j+1} r) + C(d)\kappa^{-d/p}\phi(x,\kappa^j r)\big).
\end{equation*}
Now taking summation in $j$, and using the property
$$\vq_{x,\kappa^j r} \rightarrow D^2u(x) \quad \text{as }j \rightarrow \infty$$
as noted in Section \ref{sec-mosc-minimizer}, we have
\begin{equation*}
\begin{split}
&\abs{D^2u(x)-\vq_{x,r}} \leq \sum_{j=0}^\infty \abs{\vq_{x,\kappa^{j+1} r} - \vq_{x,\kappa^j r}}\leq C(d,p,\kappa) \sum_{j=0}^{\infty} \phi(x,\kappa^j r)\\
&\lesssim C_1(d,p,\kappa)\phi(x,r) + C_2\big((\norm{D^2u}_{L_\infty(\Omega)} + \norm{Du}_{L_\infty(\Omega)}+\norm{Dg}_{L_\infty(\Omega)})I\widetilde{\omega}^{(1)}(r) + I\widetilde{\omega}^{(2)}(r)\big),
\end{split}
\end{equation*}
where $C_2=C_2(d,p,\lambda,\delta,\norm{\beta}_{C^1}, \norm{D\psi_0}_{L_\infty(\Omega)},\rho_{D\psi_0},\kappa)$. The last inequality follows from \eqref{est-mean-osc-j-decay} and Lemma \ref{lem-induced-Dini-function}, noting that both $\omega^{(1)}$ and $\omega^{(2)}$ satisfy \eqref{eq4.58}. Using the interpolation inequality, we can further deduce that
\begin{align}\label{est-est-diff-u-average-final}
&\abs{D^2u(x)-\vq_{x,r}}\nonumber\\
&\leq C_1\phi(x,r) + C_2\big((\norm{D^2u}_{L_\infty(\Omega)} + \norm{Du}_{L_2(\Omega)}+\norm{Dg}_{L_\infty(\Omega)})I\widetilde{\omega}^{(1)}(r) + I\widetilde{\omega}^{(2)}(r)\big).
\end{align}
By the definition of $\phi(x,r)$, H\"older's inequality, and \eqref{est-nbhd-equiv-ball}, we get
\begin{equation}\label{est-phi-r}
\phi(x,r) \leq \big(\fint_{\Omega_r(x)}\abs{D^2u}^p\big)^{1/p} \leq \big(\fint_{\Omega_r(x)}\abs{D^2u}^2\big)^{1/2} \lesssim r^{-d/2} \norm{D^2 u}_{L_2(\Omega)}.
\end{equation}
Similarly,
\begin{equation}\label{est-vq}
\abs{\vq_{x,r}} \leq 2^{1/p-1} \big(\phi(x,r) + \big(\fint_{\Omega_r(x)}\abs{D^2u}^p\big)^{1/p}\big) \lesssim r^{-d/2} \norm{D^2 u}_{L_2(\Omega)},
\end{equation}
where the first inequality follows by taking the average for $z\in \Omega_r(x)$ on both sides of
\begin{equation*}
\abs{\vq_{x,r}}^p \leq \abs{\vq_{x,r} - D^2u(z)}^p + \abs{D^2u(z)}^p.
\end{equation*}
Using the triangular inequality, \eqref{est-phi-r}, and \eqref{est-vq}, we can derive from \eqref{est-est-diff-u-average-final} that
\begin{align*}
&\abs{D^2u(x)} \leq Cr^{-d/2}\norm{D^2u}_{L_2(\Omega)} \nonumber\\
&\quad + C_2\big((\norm{D^2u}_{L_\infty(\Omega)} + \norm{Du}_{L_2(\Omega)}+\norm{Dg}_{L_\infty(\Omega)})I\widetilde{\omega}^{(1)}(r) + I\widetilde{\omega}^{(2)}(r)\big).
\end{align*}
Because $\Omega$ is bounded and $D^2u\in C(\overline{\Omega})$, we can find some point $\overline{x}\in\overline{\Omega}$ such that
\begin{equation*}
\abs{D^2 u(\overline{x})} = \norm{D^2u}_{L_\infty(\Omega)}.
\end{equation*}
Since $\widetilde{\omega}^{(1)}$ is a Dini function, we can choose $r$ small enough (denoted by $\overline{r}$) such that
$$C_2I\widetilde{\omega}^{(1)}(\overline{r}) \leq 1/2$$
to absorb $\norm{D^2u}_{L_\infty(\Omega)}$ term. Finally we reach
\begin{equation}\label{est-sup-D2u}
\norm{D^2u}_{L_\infty(\Omega)} \leq C(\norm{u}_{W^2_2(\Omega)} +\norm{Dg}_{L_\infty(\Omega)}+ I\widetilde{\omega}^{(2)}(\overline{r})).
\end{equation}

Next for any $x,y\in \Omega$, $x\neq y$, we estimate $D^2u(x)-D^2u(y)$. Let $r=\abs{x-y}$. Due to \eqref{est-sup-D2u}, we only need to focus on the case $r<R_0/2$. As before, we take the minimizers $\vq_{x,2r}$ and $\vq_{y,r}$ for $\phi(x,2r)$ and $\phi(y,r)$. By the triangular inequality
\begin{align*}
&\abs{D^2u(x)-D^2u(y)}^p \\
&\leq \abs{D^2u(x)-\vq_{x,2r}}^p + \abs{D^2u(y)-\vq_{y,r}}^p + \abs{\vq_{x,2r}-D^2u(z)}^p + \abs{\vq_{y,r}-D^2u(z)}^p.
\end{align*}
Taking the average for $z\in \Omega_r(y)$, noting that $\Omega_r(y)\subset\Omega_{2r}(x)$, we have
\begin{align*}
&\abs{D^2u(x)-D^2u(y)}^p \\
&\leq \abs{D^2u(x)-\vq_{x,2r}}^p + \abs{D^2u(y)-\vq_{y,r}}^p + C(d,\delta/M)\phi^p(x,2r) + \phi^p(y,r).
\end{align*}
Now we apply \eqref{est-est-diff-u-average-final}, \eqref{est-sup-D2u}, and the generalized triangular inequality to obtain
\begin{align}
                    \label{est-D^2-diff-mid}
&\abs{D^2u(x)-D^2u(y)} \nonumber\\
&\lesssim \phi(x,2r) + \phi(y,r) + \big((\norm{u}_{W^2_2(\Omega)} +\norm{Dg}_{L_\infty(\Omega)} + I\widetilde{\omega}^{(2)}(\overline{r}))I\widetilde{\omega}^{(1)}(r) + I\widetilde{\omega}^{(2)}(r)\big).
\end{align}
Let $j$ be the integer such that
$$2r \in (\kappa^{j+1}R_0, \kappa^j R_0].$$
From \eqref{est-mean-osc-j-decay}, interpolation inequalities, \eqref{est-sup-D2u}, and \eqref{est-nbhd-equiv-ball}, we obtain
\begin{align*}
&\phi(x, 2r) \\
&\leq (\frac{1}{2})^j \phi(x,2\kappa^{-j}r) + C(\norm{u}_{W^2_2(\Omega)} +\norm{Dg}_{L_\infty(\Omega)}+ I\widetilde{\omega}^{(2)}(\overline{r}))\widetilde{\omega}^{(1)}(2\kappa^{-1}r) + C\widetilde{\omega}^{(2)}(2\kappa^{-1}r)\\
&\lesssim (\frac{r}{R_0})^{\alpha_0} \phi(x,2\kappa^{-j}r) + (\norm{u}_{W^2_2(\Omega)} +\norm{Dg}_{L_\infty(\Omega)}+ I\widetilde{\omega}^{(2)}(\overline{r}))\widetilde{\omega}^{(1)}(2\kappa^{-1}r) + \widetilde{\omega}^{(2)}(2\kappa^{-1}r)\\
&\lesssim (\frac{r}{R_0})^{\alpha_0} R_0^{-d/2}\norm{D^2u}_{L_2(\Omega)} + (\norm{u}_{W^2_2(\Omega)} +\norm{Dg}_{L_\infty(\Omega)}+ I\widetilde{\omega}^{(2)}(\overline{r}))\widetilde{\omega}^{(1)}(2\kappa^{-1}r)\\ &\qquad+ \widetilde{\omega}^{(2)}(2\kappa^{-1}r),
\end{align*}
where $\alpha_0 = \log(2)/\log(1/\kappa)>0$. Similarly, we can obtain the decay rate of $\phi(y,r)$. Substituting these into \eqref{est-D^2-diff-mid} and using Lemma \ref{lem-induced-Dini-function} to bound $\widetilde{\omega}$ by $I\widetilde{\omega}$, we obtain that for any $\abs{x-y}<R_0/2$,
\begin{equation}\label{est-moc-D2u}
\begin{split}
&\abs{D^2u(x)-D^2u(y)} \lesssim \norm{D^2u}_{L_2(\Omega)}\abs{x-y}^{\alpha_0}\\
&\quad + (\|u\|_{W^2_2(\Omega)} +\norm{Dg}_{L_\infty(\Omega)}+ I\widetilde{\omega}^{(2)}(\overline{r}))I\widetilde{\omega}^{(1)}(\abs{x-y}) + I\widetilde{\omega}^{(2)}(\abs{x-y}).
\end{split}
\end{equation}
Clearly, the right-hand-side goes to zero as $r$ goes to zero, which gives us the desired estimate for the modulus of continuity of $D^2u$.

Now it remains to remove the assumption $u\in C^2(\overline{\Omega})$. For this we consider the mollified problem:
\begin{equation*}
\begin{cases}
-a_{ij}^{(n)}D_{ij}u_n + u_n = f^{(n)} + u &\text{in}\,\Omega^{(n)},\\
\beta^{(n)}\cdot Du_n = g^{(n)} &\text{on}\,\p\Omega^{(n)},
\end{cases}
\end{equation*}
where for some fixed $\alpha\in(0,1)$,
\begin{equation*}
\Omega^{(n)}\nearrow\Omega, \quad \Omega^{(n)},\vec{\beta}^{(n)}, g^{(n)}\in C^{1,\alpha},\quad a_{ij}^{(n)}, f^{(n)}\in C^\alpha
\end{equation*}
with corresponding moduli of continuity (either in the $L_1$ or $L_\infty$ sense), which are uniform with respect to $n$. Note that as mentioned before, by the $W^2_q$ well-posedness, bootstrap, and the Sobolev embedding, we can derive that any $W^2_2$-strong solution $u$ is also $C^\alpha$. According to \cite{Sa}, we can find a unique solution $u_n\in C^{2,\alpha}(\Omega^{(n)})$ for each $n$. Now, using \eqref{est-sup-D2u}, \eqref{est-moc-D2u}, the Arzela-Ascoli theorem, and a diagonal argument, we can obtain a subsequence which converges in $C^2(\overline{\Omega^{(k)}})$ for every $k$. Clearly the limit $u_\infty \in C^2(\overline{\Omega})$ and satisfies the equation. To see that $u_\infty$ satisfies the boundary condition, we extend
$$
G^{(n)}:={\vec{\beta}}^{(n)}\cdot Du_n-g^{(n)}\in \mathring{W}^1_2(\Omega^{(n)})$$
to be zero outside $\Omega^{(n)}$, so that $G^{(n)} \in \mathring{W}^1_2(\Omega)$. Since the $W^{1}_2(\Omega)$ norm of $G^{(n)}$ is uniformly bounded, by passing to a further subsequence and noting that $\mathring{W}^1_2(\Omega)$ is weakly closed in $W^1_2(\Omega)$, we obtain $\vec{\beta}\cdot Du_\infty-g=0$. Due to the uniqueness of strong solutions in small Lipschitz domains proved in \cite{DL}, $u_\infty=u$. Hence any $W^2_2$-strong solution must also be $C^2(\overline{\Omega})$. This finishes the proof of Theorem \ref{thm-main}.

\section{Nonlinear Equations}
Our method can also be applied to derive the $C^2(\overline{\Omega})$ regularity for fully nonlinear equations. In this section, we prove Theorem \ref{thm-nonlinear}. As preparation, we first introduce two lemmas, which can be viewed as the nonlinear version of Lemma \ref{lem-C3} and Corollary \ref{cor-strong-p-1}, as well as an interpolation inequality. The first lemma deals with the function $F:\bS^{d\times d}\rightarrow \bR$, satisfying
$$F \,\,\text{is concave}, \quad\lambda \norm{M}\leq F(M+N)-F(N)\leq \lambda^{-1}\norm{M},\,\,\forall M,N\in \bS^{d\times d},\,\text{and}\,M\geq 0.$$
\begin{lemma}\label{lem-C2alpha-mix}
For any continuous function $\varphi$ and constant $C$, there exists a unique solution $u$ in $C^2(B_1^{+}\cup\Sigma_1)\cap C(\overline{B_1^{+}})$ to
\begin{equation}\label{eqn-mixed-nonlinear}
F(D^2u)=C \,\,\text{in}\,B_1^{+},\quad \frac{\p u}{\p x^d}=0 \,\,\text{on}\, \Sigma_1,\quad u=\varphi \,\,\text{on}\, \sigma_1.
\end{equation}
Furthermore, there exists some constant $\overline{\alpha}=\overline{\alpha}(d,\lambda)\in(0,1)$, such that $u\in C^{2,\overline{\alpha}}_{loc}(B_1^+\cup\Sigma_1)$, and
\begin{equation}\label{est-C2alpha-mixed}
[D^2u]_{C^{\overline{\alpha}}(B_{1/2}^{+})} \leq N(d,\lambda,p)(|D^2u -\vec{q}|^p)^{1/p}_{B_1^{+}}
\end{equation}
holds for any $p\in (0,1)$ and any constant matrix $\vec{q}\in\bS^{d\times d}$.
\end{lemma}
\begin{proof}
The unique solvability of \eqref{eqn-mixed-nonlinear} and the following boundary estimate of the Evans-Krylov type for Neumann problem are classical:
\begin{equation}\label{est-Holder-nonlinear-original}
[D^2u]_{C^{\overline{\alpha}}(B_{1/2}^{+})} \leq N(d,\lambda)\norm{u}_{L_2(B_1^{+})}.
\end{equation}
See, for example, \cite[Theorem~8.1]{Sa-notes}.
Now we prove
\begin{equation}\label{est-C2alpha-2}
[D^2u]_{C^{\overline{\alpha}}(B_{1/2}^{+})} \leq N(d,\lambda)(|D^2u -\vec{q}|^2)^{1/2}_{B_1^{+}},\quad \forall \vec{q}\in\bS^{d\times d}.
\end{equation}
For this, consider
$$v(x):=u(x)-\langle x,\overline{\vec{q}}x\rangle/2 -l(x'),\quad \overline{\vec{q}}:=\begin{bmatrix}
\vec{q}'&0\\0&q_{dd}
\end{bmatrix}.$$
where $\vec{q}'\in \bS^{(d-1)\times(d-1)}$ is the first $(d-1)\times (d-1)$ submatrix of $\vec{q}$, $x'=(x_1,\ldots,x_{d-1})$, and $l(x')$ is the affine function chosen suitably to make $(v)_{B_{3/4}^{+}} = (D_{x'}v)_{B_{3/4}^{+}}=0$.
Then $D^2v=D^2u-\overline{\vec{q}}$, and $v$ satisfies
\begin{equation*}
G(D^2v):=F(D^2v+\overline{\vec{q}})=C \,\,\text{in}\,B_1^{+},\quad \frac{\p v}{\p x^d}=0 \,\,\text{on}\, \Sigma_1.
\end{equation*}
From a rescaled version of \eqref{est-Holder-nonlinear-original} for $v$, the Sobolev-Poincar\'e inequality, and the boundary Poincar\'e inequality, we obtain
\begin{align}
[D^2v]_{C^{\overline{\alpha}}(B_{1/2}^{+})} &\lesssim (|v|^2)^{1/2}_{B_{3/4}^{+}}\nonumber\\
&\lesssim (|D_{x'}v|^2)^{1/2}_{B_{3/4}^{+}} + (|D_dv|^2)^{1/2}_{B_{3/4}^{+}}\lesssim (|DD_{x'}v|^2)^{1/2}_{B_{3/4}^{+}} + (|D_d^2v|^2)^{1/2}_{B_{3/4}^{+}}.\label{est-holder-halfspace-u-v}
\end{align}
We can remove the $D_{x'}D_dv$ term from the right-hand side. Indeed, by applying the boundary $W^1_2$ estimate of the Dirichlet problem
\begin{equation*}
-\Delta(D_dv)=\dv(-(0,\ldots,\Delta v))\,\,\text{in}\,B_1^{+},\quad D_dv =0\,\,\text{on}\, \Sigma_1,
\end{equation*}
we have
\begin{equation*}
(|DD_dv|^2)^{1/2}_{B_{3/4}^{+}}\lesssim (|\Delta v|^2)^{1/2}_{B_1^{+}} + (|D_dv|^2)^{1/2}_{B_1^{+}}\lesssim (|\Delta v|^2)^{1/2}_{B_1^{+}} + (|D^2_dv|^2)^{1/2}_{B_1^{+}},
\end{equation*}
where in the last inequality, again we use the boundary Poincar\'e inequality. Substituting this into \eqref{est-holder-halfspace-u-v}, we obtain \eqref{est-C2alpha-2}:
\begin{equation*}
\begin{split}
[D^2u]_{C^{\overline{\alpha}}(B_{1/2}^{+})} = [D^2v]_{C^{\overline{\alpha}}(B_{1/2}^{+})} &\lesssim (|D_{x'}^2 v|^2)^{1/2}_{B_1^{+}} + (|D_d^2 v|^2)^{1/2}_{B_1^{+}} \\
&= (|D_{x'}^2 u - \overline{\vec{q}}|^2)^{1/2}_{B_1^{+}} + (|D_d^2 u - q_{dd}|^2)^{1/2}_{B_1^{+}} \leq (|D^2u -\vec{q}|^2)^{1/2}_{B_1^{+}}.
\end{split}
\end{equation*}
From \eqref{est-C2alpha-2}, we obtain \eqref{est-C2alpha-mixed} using standard scaling and iteration argument as mentioned in the proof of Lemma \ref{lem-C3}. Notice that the corresponding interior version of \eqref{est-C2alpha-2} can be obtained by a similar technique by applying the interior Evans-Krylov estimate and the Sobolev-Poincar\'e inequality to the function
$$
v(x):=u(x)-\langle x,qx\rangle/2 - l(x),
$$
where $l(x)$ is the affine function such that $(v)_{B_{3/4}} = (Dv)_{B_{3/4}}=0$. The lemma is proved.
\end{proof}
The second lemma is a boundary $W^{2}_{\epsi}$-estimate in the spirit of \cite{Lin}. See, for example, \cite[pp.~19]{DL}.
\begin{lemma}\label{lem-w2epsi}
Suppose that $u\in W^2_d(B_1^{+})$ and $f\in L_d(B_1^{+})$ satisfy
\begin{equation}\label{eqn-theta-nonlinear}
a_{ij}D_{ij}u = f\,\,\text{in}\,B_1^{+},\quad \frac{\p u}{\p x^d}=0 \,\,\text{on}\, \Sigma_1,\quad u=0 \,\,\text{on}\, \sigma_1,
\end{equation}
where $a_{ij}$ is symmetric, bounded measurable, and uniformly elliptic with constant $\lambda$. Then there exists some $\epsi=\epsi(d,\lambda)\in(0,1)$, such that
\begin{equation*}
\norm{D^2u}_{L_\epsi(B_1^{+})} \le N(d,\lambda) \norm{f}_{L_d(B_1^{+})}.
\end{equation*}
\end{lemma}
Recall the notation $\rho_{D^2u}(r):=\sup_{|x-y|\leq r, x,y\in \Omega}|D^2u(x)-D^2u(y)|$. Proceeding as in \cite[Lemma~3.1.4,~Theorem~3.2.1]{K-book}, we have the following interpolation inequality.
\begin{lemma}\label{lem-interpolation}
Let $\Omega$ be a domain in $\bR^d$ satisfying the interior cone condition with opening $\theta(\Omega)>0$ and height $h(\Omega)>0$. Then for any $u\in C^2(\overline{\Omega})$ and $\tau\in (0,\frac{\theta(\Omega)}{\theta(\Omega)+\pi}h(\Omega))$,
\begin{equation*}
\norm{D^2u}_{L_\infty(\Omega)} \leq C(d,\theta(\Omega))(\rho_{D^2u}(\tau) + \tau^{-2}\norm{u}_{L_\infty(\Omega)}).
\end{equation*}
\end{lemma}

Now we turn to the proof of Theorem \ref{thm-nonlinear}. This is similar to that of Theorem \ref{thm-main}, which we will sketch here.
\begin{proof}[Proof of Theorem \ref{thm-nonlinear}]
The proof is spitted into several steps. We first derive the a priori estimates for $\rho_{D^2u}$ corresponding to \eqref{est-moc-D2u}, assuming $u\in C^2(\overline{\Omega})$. Then we use the interpolation and the Aleksandrov-Bakel'man-Pucci (ABP) maximum principle in \cite[Theorem~6.1]{Lie-book} to obtain the estimate for $\norm{D^2u}_{L_\infty(\Omega)}$ and remove all the $u$ terms on the right-hand side of the estimates. Lastly, we construct a $C^{2,\alpha}$-approximating sequence. Using the uniform estimates, we can show that the limit exists, solves the problem, and is in $C^2(\overline{\Omega})$.

{\em Step 1: The a priori estimate}. The key step is to derive the $L_\epsi$-mean oscillation estimate corresponding to \eqref{est-u-mean-Dini}: for any $x\in\overline{\Omega}, \kappa \in (0,1)$, and $r\in (0, R_0]$,
\begin{align}
&\phi(x,\kappa r) \nonumber\\
            \label{est-mean-osc-nonlinear}
&\leq N\kappa^{\overline{\alpha}}\phi(x,r) + N(\kappa)\big((\norm{D^2u}_{L_\infty(\Omega)} + \norm{Du}_{L_\infty(\Omega)} + \norm{u}_{L_\infty(\Omega)})\omega^{(3)}(r) + \omega^{(4)}(r)\big).
\end{align}
Here $\phi$ is defined in \eqref{eqn-def-Lp-mean-osc} with $p$ replaced by $\epsi$ which is given in Lemma \ref{lem-w2epsi}, and $\overline{\alpha}$ is introduced in Lemma \ref{lem-C2alpha-mix}. The Dini functions $\omega^{(3)}$ and $\omega^{(4)}$ are defined as follows
\begin{equation*}
\begin{split}
&\omega^{(3)}(r):= r(\norm{D\beta}_{L_\infty}+1) + \rho_{D\beta}(r) + \rho_{D\psi_0}(r) + \omega_F(r),\\
&\omega^{(4)}(r):= r\norm{Dg}_{L_\infty} + \rho_{Dg}(r) + \rho_{D\psi_0}(r)\norm{Dg}_{L_\infty} + \omega_F(r).
\end{split}
\end{equation*}
Clearly, it suffices to prove \eqref{est-mean-osc-nonlinear} for two cases: $x\in\p\Omega$ or $B_r(x)\subset\Omega$. We only focus on the first one, since the same argument below dealing with the Neumann problem in half balls will still work for the interior case.  As before, we take the coordinates $y$ centered at $x$.

For $r\leq R_0$, as before, we find $w\in W^2_d(\Omega)$ solving \eqref{eqn-def-w}, as well as the parabola $h$ defined in \eqref{eqn-def-h}. Note that according to \cite{DL}, such solution exists, and the following $L_d$-version of \eqref{est-w}
\begin{equation}\label{est-w-Ld}
(\abs{D^2w}^d)^{1/d}_{\Omega_{r/2}}\leq C\big((r\norm{D\beta}_{L_\infty}+\rho_{D\beta}(r))(\norm{D^2u}_{L_\infty(\Omega_r)} + \norm{Du}_{L_\infty(\Omega_r)}) + r\norm{Dg}_{L_\infty} + \rho_{Dg}(r)\big)
\end{equation}
holds. Observe that the right-hand side is a Dini function. Now $v:=u-w-h$ satisfies
\begin{equation}\label{eqn-v-nonlinear}
\begin{cases}
G_0(D^2v(y))= f(y) - F_0(D^2h,(Du)_{\Omega_r}, (u)_{\Omega_r}) &\text{in}\,\Omega_r,\\
\p v/\p y_d=0\quad &\text{on}\,\Gamma_r,
\end{cases}
\end{equation}
where $F_0$ is the function chosen in Assumption \ref{ass-nonlinear} and
\begin{equation*}
\begin{split}
&G_0(M):=F_0(M+D^2h,(Du)_{\Omega_r}, (u)_{\Omega_r}) - F_0(D^2h,(Du)_{\Omega_r}, (u)_{\Omega_r}),\\
&f(y):=F_0(D^2u(y)-D^2w(y),(Du)_{\Omega_r}, (u)_{\Omega_r}).
\end{split}
\end{equation*}
As in the linear case, we first prove the mean oscillation estimate for $D^2v$,
\begin{equation}\label{est-mean-osc-v-nonlinear}
\begin{split}
&(\fint_{\Omega_{\kappa r}}|D^2v -\vec{V}_{\kappa r}|^\epsi)^{1/\epsi}
\lesssim \kappa^{\overline{\alpha}} (\fint_{\Omega_r}|D^2v -\vec{q}|^\epsi)^{1/\epsi} \\
&\quad + C(\kappa) \left(\omega^{(3)}(r)(\norm{D^2u}_{L_\infty}+\norm{Du}_{L_\infty}+\norm{u}_{L_\infty})+\omega^{(4)}(r)\right),
\end{split}
\end{equation}
which can be compared to \eqref{est-mean-osc-homogeneous}.
To prove this, we flatten the boundary for the problem \eqref{eqn-v-nonlinear} using the change of variables \eqref{eqn-cov}. For $z\in B^{+}_{(\delta/M)r}$, the equation becomes
\begin{equation}\label{eqn-v-flat-nonlinear}
\widetilde{G}_0(D^2_zv):=G_0((\frac{\p z}{\p y}(0))^TD^2_zv\frac{\p z}{\p y}(0)) = \widetilde{f} - F_0(D^2h,(Du)_{\Omega_r},(u)_{\Omega_r}),
\end{equation}
where
$$\widetilde{f}:= G_0((\frac{\p z}{\p y}(0))^TD^2_zv\frac{\p z}{\p y}(0)) -G_0((\frac{\p z}{\p y})^TD^2_zv\frac{\p z}{\p y} + \frac{\p v}{\p z_d}D^2\psi) + f.$$
We decompose $v=\theta+\xi$, where $\xi\in C^2(B^{+}_{(\delta/M)r}\cup\Sigma_{(\delta/M)r})\cap C(\overline{B^{+}_{(\delta/M)r}})$ solves
\begin{equation}\label{eqn-xi-nonlinear}
\begin{cases}
\widetilde{G}_0(D^2_z\xi)= - F_0(D^2h,(Du)_{\Omega_r},(u)_{\Omega_r}) &\text{in}\,B^{+}_{(\delta/M)r},\\
\p\xi/\p z_d=0 &\text{on}\,\Sigma_{(\delta/M)r},\\
\xi=v &\text{on}\,\sigma_{(\delta/M)r}.
\end{cases}
\end{equation}
Such $\xi$ exists and satisfies the boundary $C^{2,\overline{\alpha}}$ estimate according to Lemma \ref{lem-C2alpha-mix}. Recalling \eqref{eqn-inclusion}, we can further deduce from a rescaled version of \eqref{est-C2alpha-mixed} that for any $\widetilde{\vec{q}}\in\bS^{d\times d}$,
\begin{equation}\label{est-xi-nonlinear}
\big(\fint_{B_{2\kappa r}^{+}} \abs{D^2\xi - (D^2\xi)_{B_{2\kappa r}^{+}}}^\epsi\big)^{1/\epsi} \lesssim (2\kappa r)^{\overline{\alpha}}[D^2\xi]_{C^{\overline{\alpha}}(B_{2\kappa r}^{+})} \lesssim\kappa^{\overline{\alpha}} (\abs{D^2\xi - \widetilde{\vec{q}}}^p)^{1/p}_{B_{(\delta/M)r}^{+}}.
\end{equation}
Taking the difference between \eqref{eqn-v-flat-nonlinear} and the first line of \eqref{eqn-xi-nonlinear}, noting that $\widetilde{G}_0(0)=G_0(0)=0$, we see that $\theta$ satisfies \eqref{eqn-theta-nonlinear} with $\lambda$, $f$, and $B_1^+$ replaced by $\lambda/d$, $\widetilde{f}$, and $B_{\delta/M}^+$. By Lemma \ref{lem-w2epsi},
\begin{equation}\label{est-theta-nonlienar}
(\abs{D^2\theta}^\epsi)^{1/\epsi}_{B_{(\delta/M)r}^{+}} \lesssim (\abs{\widetilde{f}}^d)^{1/d}_{B_{(\delta/M)r}^{+}}.
\end{equation}
Using $F[u]=0$, Assumption \ref{ass-nonlinear}, Hardy's inequality, and \eqref{est-w-Ld}, we can estimate $\widetilde{f}$ as follows
\begin{align*}
(|\widetilde{f}|^d)^{1/d}_{B_{(\delta/M)r}^{+}} \lesssim& (\abs{G_0((\frac{\p z}{\p y}(0))^TD^2_zv\frac{\p z}{\p y}(0))-G_0((\frac{\p z}{\p y})^TD^2_zv\frac{\p z}{\p y} + \frac{\p v}{\p z_d}D^2\psi)}^d)^{1/d}_{B_{(\delta/M)r}^{+}}\\
&+(\abs{F_0(D^2u-D^2w,(Du)_{\Omega_r}, (u)_{\Omega_r}) - F_0(D^2u,(Du)_{\Omega_{r}}, (u)_{\Omega_r})}^d)^{1/d}_{\Omega_{r/2}}\\
&+(\abs{F_0(D^2u,(Du)_{\Omega_r}, (u)_{\Omega_r}) - F(D^2u,(Du)_{\Omega_r}, (u)_{\Omega_r},y)}^d)^{1/d}_{\Omega_{r/2}}\\
&+(\abs{F(D^2u,(Du)_{\Omega_r}, (u)_{\Omega_r},y) - F(D^2u,Du,u,y)}^d)^{1/d}_{\Omega_{r/2}}
\\
\lesssim& \rho_{D\psi_0}(r)(|D^2v|^d)^{1/d}_{\Omega_{r/2}} + (|D^2w|^d)^{1/d}_{\Omega_{r/2}}\\
&+ \omega_F(r)(\norm{D^2u}_{L_\infty}+\norm{Du}_{L_\infty}+\norm{u}_{L_\infty}+1) + r(\norm{D^2u}_{L_\infty}+\norm{Du}_{L_\infty})\\
\lesssim& \omega^{(3)}(r)(\norm{D^2u}_{L_\infty}+\norm{Du}_{L_\infty}+\norm{u}_{L_\infty}) + \omega^{(4)}(r).
\end{align*}
Combining \eqref{est-xi-nonlinear} and \eqref{est-theta-nonlienar}, and following the proof of Lemma \ref{lem-homogeneous}, we obtain \eqref{est-mean-osc-v-nonlinear}. Then, the same steps as in the proof of Proposition \ref{prop-mean-oscillation} leads to \eqref{est-mean-osc-nonlinear}. The iteration argument in the proof of Theorem \ref{thm-main} gives, for any $r\in (0,R_0/2)$,
\begin{equation}
\rho_{D^2 u}(r)
\leq C\big(\norm{D^2u}_{L_\infty(\Omega)}r^{\alpha_0} + (\norm{D^2u}_{L_\infty(\Omega)} + \norm{u}_{L_\infty(\Omega)})I\widetilde{\omega}^{(3)}(r) + I\widetilde{\omega}^{(4)}(r)\big),\label{est-apriori-moc-D2u}
\end{equation}
where as before, $\alpha_0=\log(2)/\log(1/\kappa)>0$.

{\em Step 2: Interpolation and maximum principle.} In this step, we aim to bound $\norm{D^2u}_{L_\infty(\Omega)}$ and remove all the $u$ terms from the right-hand side of \eqref{est-apriori-moc-D2u}. Noting \eqref{eqn-def-R0}, we can choose the parameters
$$\theta(\Omega)=2\arcsin(\delta/2),\quad h(\Omega)=\sqrt{3}R_0/2$$
 for the interior cone at each point.
Using Lemma \ref{lem-interpolation} and \eqref{est-apriori-moc-D2u}, we obtain, for $0<\tau < \sqrt{3}R_0\frac{\arcsin(\delta/2)}{2\arcsin(\delta/2)+\pi}$,
\begin{equation*}
\begin{split}
\norm{D^2u}_{L_\infty(\Omega)}&\leq C(d,\delta)(\rho_{D^2u}(\tau)+\tau^{-2}\norm{u}_{L_\infty(\Omega)})\\
&\leq C\big((\tau^{\alpha_0} + I\widetilde{\omega}^{(3)}(\tau))\norm{D^2u}_{L_\infty(\Omega)} + (\tau^{-2} + I\widetilde{\omega}^{(3)}(\tau)) \norm{u}_{L_\infty(\Omega)} + I\widetilde{\omega}^{(4)}(\tau)\big).
\end{split}
\end{equation*}
Now, choose $\tau>0$ sufficiently small to absorb the first term on the right-hand side. Then, noting the sign condition \eqref{cond-sign}, we can use the ABP estimate in \cite[Theorem~6.1]{Lie-book} to bound $\norm{u}_{L_\infty(\Omega)}$. This leads to
\begin{equation}\label{est-nonlinear-D2u-sign}
\norm{D^2u}_{L_\infty(\Omega)}\lesssim \norm{F(0,0,0,\cdot)}_{L_d(\Omega)} + \norm{g}_{L_\infty(\Omega)} + I\widetilde{\omega}^{(4)}(1).
\end{equation}
Substituting back into \eqref{est-apriori-moc-D2u} and using the ABP estimate again, we conclude
\begin{equation}\label{est-nonlinear-moc-D2u-sign}
\rho_{D^2u}(r)\lesssim (\norm{F(0,0,0,\cdot)}_{L_d(\Omega)} + \norm{g}_{L_\infty}(\Omega) + I\widetilde{\omega}^{(4)}(1))(r^{\alpha_0} + I\widetilde{\omega}^{(3)}(r)) + I\widetilde{\omega}^{(4)}(r).
\end{equation}

{\em Step 3: Approximation.} We fix some constant $\alpha\in(0,\overline{\alpha})$, and take the mollification
\begin{equation*}
\Omega^{(n)}\nearrow\Omega, \quad \Omega^{(n)},\beta^{(n)}_0, \vec{\beta}^{(n)}, g^{(n)}\in C^{1,\alpha},\quad [F^{(n)}(M,p,u,\cdot)]_\alpha \leq K_n(|M|+|p|+|u|+1),
\end{equation*}
with corresponding moduli of continuity (in the sense of $L_\infty$ or Assumption \ref{ass-nonlinear}), which are uniform with respect to $n$. According to \cite[Theorem~3.3]{Sa-notes}, for each $n$, there exists a unique solution $u_n\in C^{2,\alpha}(\Omega^{(n)})$ to
\begin{equation*}
F^{(n)}(D^2u,Du,u,y)=0\,\,\text{in}\,\Omega^{(n)},\quad B^{(n)}u=g^{(n)}\,\,\text{on}\,\p\Omega^{(n)}.
\end{equation*}
Notice that \eqref{est-nonlinear-D2u-sign} and \eqref{est-nonlinear-moc-D2u-sign} give us the $C^2(\overline{\Omega^{(k)}})$ pre-compactness of the family $\{u^{(n)}\}_{n\geq k}$. Similar compactness argument as in the linear case gives us the unique $C^2(\overline{\Omega})$ solution to the original problem.
\end{proof}


\end{document}